\documentclass{amsart}

\usepackage{amssymb}
\usepackage{amsmath}
\usepackage{amsthm}

\newtheorem{theorem}{Theorem}[section]
\newtheorem {prop}[theorem]{Proposition}
\newtheorem {cory}[theorem]{Corollary}
\newtheorem {lem}[theorem]{Lemma}

\theoremstyle{definition}

\theoremstyle{remark}
\newtheorem{rem}{Remark}[section]

\numberwithin{equation}{section}

\newcommand{\ph}{\varphi}
\newcommand{\al}{\alpha}
\newcommand{\er}{\varepsilon}
\newcommand{\za}{\zeta}
\newcommand{\Om}{\Omega}
\newcommand{\om}{\omega}
\newcommand{\wb}{\overline{w}}
\newcommand{\hol}{\mathrm{H}}
\newcommand{\lip}{\Lambda}

\newcommand{\Cbb}{\mathbb C}
\newcommand{\Dbb}{\mathbb D}
\newcommand{\Nbb}{\mathbb N}
\newcommand{\zz}{{\mathbb Z}_+}
\newcommand{\Rbb}{\mathbb R}

\newcommand{\cph}{C_\varphi}

\numberwithin{equation}{section}

\begin{document}

\title[Essential norms of weighted composition operators]
{Essential norms of weighted composition operators between Lipschitz spaces of arbitrary order}

\author{Evgueni Doubtsov}

\address{St.~Petersburg Department
of V.A.~Steklov Mathematical Institute,
Fontanka 27, St.~Petersburg 191023, Russia}

\email{dubtsov@pdmi.ras.ru}

\subjclass{Primary 47B33; Secondary 30H30, 30H99, 46E15, 47B38}

\keywords{Holomorphic Lipschitz space, composition operator, essential norm}

\thanks{The author was partially supported by RFBR (grant No.~17-01-00607).}

\begin{abstract}
Let $\mathbb{D}$ denote the unit disk of $\mathbb{C}$ and
let $\Lambda^\alpha(\mathbb{D})$ denote the scale of holomorphic Lipschitz spaces extended to all $\alpha\in\mathbb{R}$.
For arbitrary $\alpha, \beta\in\mathbb{R}$, we characterize the bounded weighted
composition operators from $\Lambda^\beta(\mathbb{D})$ into $\Lambda^\alpha(\mathbb{D})$
and estimate their essential norms.
\end{abstract}

%%% ----------------------------------------------------------------------
\maketitle
%%% ----------------------------------------------------------------------

\section{Introduction}\label{s_int}
Let $\hol(\Dbb)$  denote the space of holomorphic functions on the
unit disk $\Dbb$ of $\Cbb$.
For $\al\in\Rbb$, the space $\lip^\al(\Dbb)$  consists of those
$f\in\hol(\Dbb)$ for which
\begin{equation*}\label{e_df_lip}
|f^{(J)}(z)|(1-|z|)^{J-\al} \le C, \quad z\in\Dbb,
\end{equation*}
where $C=C(f)$ is a constant,
$f^{(J)}$ is the derivative of order $J$, $J$  is a non-negative integer such that $J>\al$.
It is well known that the definition of the space $\lip^\al(\Dbb)$
does not depend on $J$ whenever $J>\al$.
For $\al\in\Rbb$, $J\in\zz$ and $J>\al$, $\lip^\al(\Dbb)$
is a Banach space with respect to the following norm:
\[
\|f\|_{\lip^{\al, J}(\Dbb)} =
\sum_{j=0}^{J-1} |f^{(j)}(0)| +
\sup_{z\in\Dbb}|f^{(J)}(z)|(1-|z|)^{J-\al},
\]
where the first summand is defined to be zero for $J=0$.
The above norms are equivalent for different parameters $J$, $J>\alpha$,
so in what follows, we use the brief notation
$\|\cdot\|_{\lip^\al}$ in the place of
$\|\cdot\|_{\lip^{\al, J}(\Dbb)}$ with the smallest $J$, $J>\alpha$.

\subsection*{Holomorphic Lipschitz spaces}
If $\al>0$, then each $f\in \lip^\al(\Dbb)$
extends to a Lipschitz function on the unit circle $\partial\Dbb$.
So, we say that $\lip^\al(\Dbb)$, $\al>0$, is a holomorphic Lipschitz space.
The standard holomorphic Lipschitz spaces $\lip^\al(\Dbb)$
are those with $0<\alpha<1$.
Also, $\lip^\al(\Dbb)$ with $\al<1$ are often called Bloch type spaces
because the Bloch space $\lip^0(\Dbb)$ is defined with $J=1$ (see, for example, \cite{McZ03, OSZ03}).
One has $J=2$ in the definition of the classical Zygmund space $\lip^1(\Dbb)$,
so $\lip^\al(\Dbb)$, $1\le \al< 2$ or $\al<2$,
are sometimes called Zygmund type spaces (see, for example, \cite{EL13}).

The scale $\lip^\al(\Dbb)$
splits at the point $\alpha=0$.
Indeed, $\lip^0(\Dbb)$ is the classical Bloch space;
$\lip^\al(\Dbb)$, $\al<0$, is the growth space defined by the following condition:
\[
|f(z)| \le C (1-|z|)^\alpha, \quad z\in \Dbb.
\]

\subsection*{Weighted composition operators}
Given a function $g\in\hol(\Dbb)$ and a holomorphic mapping $\ph: \Dbb \to
\Dbb$, the weighted composition operator $\cph^g: \hol(\Dbb) \to
\hol(\Dbb)$ is defined by the formula
\begin{equation*}
(\cph^g f)(z) = g(z) f(\ph(z)), \quad f\in\hol(\Dbb),\quad z\in \Dbb.
\end{equation*}
If $g\equiv 1$ then $\cph^g$ is denoted by the symbol $\cph$
and it is called a composition operator.

For $\al, \beta <1$, the bounded and compact weighted composition operators
$\cph^g: \lip^\beta(\Dbb) \to \lip^\al(\Dbb)$ were characterized by
Ohno, Stroethoff and Zhao \cite{OSZ03}.
The corresponding estimates for the essential norms were obtained by
MacCluer and Zhao \cite{McZ03}.
For arbitrary $\al, \beta\in\Rbb$, the bounded and compact operators
$\cph^g: \lip^\beta(\Dbb) \to \lip^\al(\Dbb)$
were characterized in \cite{Du11}.
See also \cite{CT16, Zo17} for general approaches to related problems.

For $\al, \beta <0$, a different approach to estimates of the corresponding essential norms
was proposed by Bonet, Doma{\'n}ski and Lindstr{\"o}m \cite{BDL99}.
They applied the theory of associated weights to obtain certain
discrete conditions in terms of the monomials $z^n$ and powers $\ph^n$, $n\in\zz$.
Further results in this direction were obtained, in particular,
by Wulan, Zheng and Zhu \cite{WZZ09} for $\al=\beta =0$,
by Zhao \cite{Z10} for $\al, \beta <1$,
and by Esmaeili and Lindstr\"om \cite{EL13} for $\al, \beta <2$.

We realize both approaches mentioned above
to estimate the essential norms of the operators $\cph^g: \lip^\beta(\Dbb) \to \lip^\al(\Dbb)$
for arbitrary $\al, \beta\in\Rbb$; see Theorem~\ref{t_ess}.
Clearly, the computational complexity of the problem under consideration increases for large $\al$ and $\beta$.
So the main issue is to find appropriate test functions for the proof of the corresponding lower estimates;
see Lemma~\ref{l_test} and Proposition~\ref{p_ess}.

\subsection*{Organization of the paper}
In Section~\ref{s_bdd}, we prove discrete versions of the function-theoretic characterizations
obtained in \cite{Du11} for the bounded operators $\cph^g: \lip^\beta(\Dbb) \to \lip^\al(\Dbb)$, $\al, \beta\in\Rbb$.
The main results related to two-sided estimates of the corresponding essential norms are presented in Section~\ref{s_ess}.

\subsection*{Notation}
As usual, $C$ denotes a constant $C>0$ whose value may change from line to line.
We write $A\lesssim B$ if $A \le CB$ for a constant $C>0$;
$A\asymp B$ means that $A\lesssim B$ and $B\lesssim A$.

\section{Bounded operators}\label{s_bdd}

\subsection{Derivatives of the weighted composition}\label{ss_der}
Let $J\in\zz$. Given an $f\in\hol(\Dbb)$,  we formally calculate the
derivative $(\cph^g f)^{(J)}$.
Clearly, the Leibniz rule gives
\[
(g f(\ph))^{(J)} = \sum_{j=0}^J {{J}\choose{j}} g^{(J-j)} (f(\ph))^{(j)}.
\]
Next, Fa\`a di Bruno's formula allows to compute the derivatives $(f(\ph))^{(j)}$, $j=0,1,\dots, J$.
We omit the combinatorial details and use the formal identity
\begin{equation*}\label{e_df_Gj}
(\cph^g f)^{(J)}(z) = \sum_{j=0}^J G_j[g, \ph, J](z)
f^{(j)}(\ph(z)), \quad z\in\Dbb,
\end{equation*}
to define the functions $G_j[g, \ph, J]$, $j=0,1,\dots, J$. In particular,
\[
\aligned
G_0[g, \ph, 0]&=g;
\\
G_0[g, \ph, 1]&=g^\prime,\quad G_1[g, \ph, 1]=g\ph^\prime;
\\
G_0[g, \ph, 2]&=g^{\prime\prime},\quad G_1[g, \ph,
2]=2g^\prime\ph^\prime + g\ph^{\prime\prime},\quad G_2[g, \ph,
2]=g(\ph^\prime)^2.
\endaligned
\]

\subsection{Weight functions for $\lip^\beta (\Dbb)$}
Given $\beta\in\Rbb$, let $N\in\zz$ be the minimal number such that $N\ge \beta$.
For $J\in\zz$, define auxiliary functions $\Om_{j, \beta}(t)$, $0\le t <1$, $j=0,1,\dots, J$,
by the following rule:

\begin{equation}\label{e_Omjb}
\begin{aligned}
& \text{if } \ N>0,\ \text{ then } \ \Om_{j, \beta} \equiv 1\ \text{ for } \
 j=0,1,\dots, \min\{N-1, J\}; \\
&\text{if } \ \beta<N\le J,\ \text{ then } \ \Om_{j, \beta}(t)=(1-t)^{\beta-j}\ \text{ for } \ j=N,\dots, J; \\
&\text{if } \ \beta=N<J,\ \text{ then } \ \Om_{j, \beta}(t)=(1-t)^{\beta-j}\ \text{ for }\ j=N+1,\dots,
J; \\
&\text{if } \ \beta=N\le J,\ \text{ then } \ \Om_{N, \beta}(t) = {\log\frac{e}{1-t}}.
\end{aligned}
\end{equation}

Observe that $\Om_{j, \beta}: [0,1)\to (0, +\infty)$ is a weight, that is, a continuous increasing unbounded function.
Given a weight $\Om$, recall that the associated weight $\widetilde{\Om}$ is defined as follows (see \cite{BDL99}):
\[
\widetilde{\Om}(t) = \sup\left\{|f(t)|: f\in\hol(\Dbb),\ |f(z)|\le\Om(|z|)  \textrm{\ for\ } z\in\Dbb\right\},
\quad 0\le t <1.
\]
It is well-known that $\Om_{j, \beta} \asymp \widetilde{\Om}_{j, \beta}$ for all parameters $j$ and $\beta$
under consideration.

\subsection{Characterizations}
Combining Propositions~2.1 and 3.1 from \cite{Du11}, we have the following characterization:

\begin{prop}\label{p_bdd}
Let $\al, \beta \in\Rbb$, $J\in\zz$ and $J>\al$.
Fix the minimal $N\in \zz$ such that $\beta\le N$.
The operator $\cph^g: \lip^\beta(\Dbb) \to \lip^\al(\Dbb)$  is bounded if and only
if
\begin{equation}\label{e_bdd}
\sup_{z\in\Dbb}\frac{|G_j[g,\ph,J](z)| \Om_{j, \beta}(|\ph(z)|)}
{(1-|z|)^{\al-J}}  < \infty,\quad
j=0,1,\dots,J,
\end{equation}
where the weight functions $\Om_{j, \beta}$ are defined by \eqref{e_Omjb}.
\end{prop}

To work with associated weights, consider the growth spaces.
Given a weight $\om: [0,1) \to (0, +\infty)$, the growth space $H^\infty_\om(\Dbb)$
consists of those $f\in\hol(\Dbb)$ for which
\[
\|f\|_{H^\infty_\om} = \sup_{z\in \Dbb} \frac{|f(z)|}{\om(|z|)} < \infty.
\]
We will need the following result:
\begin{prop}[see \cite{BDL99, Mo2k}]\label{p_zn_w}
Let $\om, \nu$ be weights on $[0,1)$, $g\in \hol(\Dbb)$ and let $\ph: \Dbb\to \Dbb$ be a holomorphic mapping.
Then
\[
\sup_{z\in \Dbb}\frac{|g(z)|\widetilde{\om}(|\ph(z)|)}{\nu(|z|)} \asymp
\sup_{n\in\zz}
\frac{\|g\ph^n\|_{H^\infty_\nu}}{\|z^n\|_{H^\infty_\om}}.
\]
\end{prop}

Recall that $\lip^\beta(\Dbb)$, $\beta<0$, is the growth space which consists
of $f\in\hol(\Dbb)$ such that $|f(z)| \le C (1-|z|)^\beta$.

\begin{cory}
Let $\al, \beta \in\Rbb$, $J\in\zz$ and $J>\al$.
Fix the minimal $N\in \zz$ such that $\beta\le N$.
The operator $\cph^g: \lip^\beta(\Dbb) \to \lip^\al(\Dbb)$  is bounded if and only
if
\begin{align}
G_j[g,\ph,J] \in \lip^{\al-J}(\Dbb),\quad
&j=0,1,\dots,\min\{N-1, J\}, \label{e_bdd_equiv}\\
 \sup_{n\in\zz}\frac{\|G_j[g,\ph,J] \ph^n\|_{\lip^{\al-J}}}
 {\|z^n\|_{H^\infty_{\Om_{j, \beta}}}}  < \infty,\quad
&j= N,\dots,J,\label{e_bdd_equiv2}
\end{align}
where the weight functions $\Om_{j, \beta}$ are defined by \eqref{e_Omjb}.
\end{cory}
\begin{proof}
For $j=0,1,\dots,\min\{N-1, J\}$, we have $\Om_{j, \beta}\equiv 1$,
so \eqref{e_bdd_equiv} coincides with \eqref{e_bdd}.
If $J\ge N$, then $\Om_{j, \beta}$, $j= N,\dots,J$, is a weight function.
  As mentioned above, $\Om_{j, \beta} \asymp \widetilde{\Om}_{j, \beta}$;
  so we apply Proposition~\ref{p_bdd} and Proposition~\ref{p_zn_w}
  with $\om = \Om_{j, \beta}$, $\nu(t) =(1-t)^{\al-J}$ and $g=G_j[g, \ph, J]$.
\end{proof}

\begin{rem}
  The norms $\|z^n\|_{H^\infty_{\Om_{j, \beta}}}$ used in \eqref{e_bdd_equiv2}
  are known up to multiplicative constants (see, for example, \cite[Lemma~2.1]{HL12}):
 \begin{align*}
  \|z^n\|_{H^\infty_{\Om_{j, \beta}}}
  &\asymp (n+1)^{\beta-j}\ \textrm{for}\ \Om_{j, \beta}(t)=(1-t)^{\beta-j},\ \beta<j; \\
  \|z^n\|_{H^\infty_{\Om_{N, N}}}
  &\asymp \frac{1}{\log (n+2)}\ \textrm{for}\ \Om_{N, N}(t)=\log \frac{e}{1-t}.
\end{align*}
\end{rem}

\section{Essential norms}\label{s_ess}

\subsection{Upper estimates}
Let $D_\beta: \lip^\beta \to \lip^{\beta-1}$ denote the
differentiation operator and let $k\in\zz$ be the smallest number
such that $k>\beta$.
Put
\[
\lip^\beta_0(\Dbb) = \{f\in\lip^\beta(\Dbb): f(0)=0\ \textrm{and}\ f^\prime(0)=\dots =f^{(k-1)}(0)=0 \ \textrm{if}\ k\ge 2\}.
\]
Then $D_\beta$ is an isometry on $\lip^\beta_0(\Dbb)$ for $k\ge 1$
and $D_\beta$ is an isomorphism on $\lip^\beta_0$ for $k=0$.

We have
\begin{align*}
  D_{\al-J+1}\dots
&D_{\al-1} D_\al \cph^g
 D_\beta^{-1} D_{\beta-1}^{-1}\dots D_{\beta-J+1}^{-1} \\
&=
\sum_{j=0}^{J-1} \cph^{G_j} D_{\beta-j}^{-1}\dots D_{\beta-J+1}^{-1}
+ \cph^{G_j},
\end{align*}
where $G_j = G_j[g,\ph,J]$.
Therefore,
\begin{equation}\label{e_ess_lem}
\|\cph^g\|_{e, \lip^\beta_0 \to \lip^\al} \lesssim
\sum_{j=0}^{J} \|\cph^{G_j}\|_{e, \lip^{\beta-j}_0 \to \lip^{\al-J}}.
\end{equation}
Observe that every compact operator $Q: \lip^\beta_0(\Dbb) \to \lip^\al(\Dbb)$
extends to a compact operator from $\lip^\beta(\Dbb)$ into $\lip^\al(\Dbb)$.
Moreover, standard arguments guarantee that
$\|\cph^g\|_{e, \lip^\beta \to \lip^\al}  = \|\cph^g\|_{e, \lip^\beta_0 \to \lip^\al}$
(see, for example, \cite[Lemma~3.1]{EL13}, where $\al, \beta <2$).
Also, by Proposition~\ref{p_bdd}, if $\cph^g: \lip^\beta(\Dbb)\to \lip^\al(\Dbb)$
is bounded, then $\cph^{G_j}: \lip^{\beta-j}(\Dbb) \to \lip^{\al-J}(\Dbb)$ is bounded,
$j=0,1,\dots, J$.
Thus, \eqref{e_ess_lem} implies the following lemma:

\begin{lem}\label{l_upper}
Let $\al, \beta \in\Rbb$, $J\in\zz$ and $J>\al$.
Assume that the weighted composition operator $\cph^g: \lip^\beta(\Dbb) \to \lip^\al(\Dbb)$
is bounded.
Then $\cph^{G_j}: \lip^{\beta-j}(\Dbb) \to \lip^{\al-J}(\Dbb)$ is bounded,
$j=0,1,\dots, J$, and
\[
\|\cph^g\|_{e, \lip^\beta \to \lip^\al} \lesssim
\sum_{j=0}^{J} \|\cph^{G_j}\|_{e, \lip^{\beta-j} \to \lip^{\al-J}}.
\]
\end{lem}

Various modifications of the following compactness criterium are well-known.

\begin{lem}\label{l_cmp}
Let $\al, \beta \in\Rbb$ and let $\cph^g: \lip^\beta(\Dbb) \to \lip^\al(\Dbb)$
be bounded.
Then $\cph^g: \lip^\beta(\Dbb) \to \lip^\al(\Dbb)$ is compact if and only if
$\|\cph^g f_n\|_{\lip^\al} \to 0$ as $n\to \infty$
for any bounded sequence $\{f_n\}_{n=1}^\infty \subset \lip^\beta(\Dbb)$
such that $f_n \to 0$ uniformly on compact subsets of $\Dbb$.
\end{lem}

Also, we will use the following standard property of sequences in $\lip^\beta(\Dbb)$.

\begin{lem}\label{l_uniform0}
Let $\beta \in\Rbb$ and let $N\in \zz$ be the minimal number such that $N\ge \beta$.
Assume that $\|f_n\|_{\lip^\beta} \le C$ and
$f_n \to 0$ uniformly on compact subsets of $\Dbb$.
If $N\ge 1$, then $f_n^{(j)}\rightrightarrows 0$ on $\Dbb$ for $j=0, 1, \dots N-1$.
\end{lem}
\begin{proof}
For $\za\in\partial\Dbb$ and $1>R>r>0$, we have
\[
|f_n^{(N-1)}(R\za)| \le |f_n^{(N-1)}(r\za)| + C \int_r^1 \Om_{\beta, N}(t)\,dt,
\]
where $\Om_{\beta, N}$ is defined by \eqref{e_Omjb} with $N = J$.
Let $\er>0$. Clearly, the above integral converges, so fix an
$r$ so close to $1$ that the integral is estimated by $\er/C$.
Since $f_n \to 0$ uniformly on compact subsets of $\Dbb$,
the derivative $f_n^{(N-1)}$ has the same property;
thus, $|f_n^{(N-1)}(r\za)|<\er$ for all sufficiently large $n$.
Therefore, $f_n^{(N-1)}\rightrightarrows 0$ on $\Dbb$, so
$f_n^{(j)}\rightrightarrows 0$ on $\Dbb$ for $j=0, 1, \dots N-1$.
\end{proof}

\begin{prop}\label{p_upper}
Let $\al, \beta \in\Rbb$, $J\in\zz$, $J>\al$ and
let $N\in \zz$ be the minimal number such that $N\ge \beta$.
Let the weight functions $\Om_{j, \beta}$, $j=0,1,\dots, J$, be those defined by \eqref{e_Omjb}.
Assume that the weighted composition operator $\cph^g: \lip^\beta(\Dbb) \to \lip^\al(\Dbb)$
is bounded.
Then
\begin{equation}\label{e_ess_up}
\|\cph^g\|_{e, \lip^\beta \to \lip^\al} \lesssim
\sum_{j=N}^J \limsup_{|\ph(z)|\to 1-}
|G_j[g,\ph,J](z)| (1-|z|)^{J-\al} {\Om_{j, \beta}(|\ph(z)|)}.
\end{equation}
In particular, $\cph^g: \lip^\beta(\Dbb) \to \lip^\al(\Dbb)$ is compact if $N>J$.
\end{prop}
\begin{proof}
We will apply Lemma~\ref{l_upper}.
So, fix $j\in \{N, N+1, \dots, J\}$ and $\delta\in (0,1)$.
Let $G_j = G_j[g,\ph,J]$. By Lemma~\ref{l_upper},
the operator $\cph^{G_j}: \lip^{\beta-j}(\Dbb) \to \lip^{\al-J}(\Dbb)$ is bounded.
Now, let $\{r_m\}_{m=1}^\infty \subset (0,1)$ be an increasing sequence converging to $1$.
For $m\in \Nbb$, the operator $C_{r_m\ph}^{G_j}: \lip^{\beta-j}(\Dbb) \to \lip^{\al-J}(\Dbb)$ is compact, hence,
\begin{equation}\label{e_upper_u}
\begin{aligned}
  \|\cph^{G_j} &\|_{e, \lip^{\beta-j} \to \lip^{\al-J}}
\le \|\cph^{G_j} - C_{r_m\ph}^{G_j}\|_{\lip^{\beta-j} \to \lip^{\al-J}} \\
&\le \sup_{|\ph(z)|\le \delta} \sup_{\|f\|_{\lip^{\beta-j}}\le 1}
(1-|z|)^{J-\al} |G_j(z)||f(\ph(z)) - f(r_m\ph(z))|\\
&\ + \sup_{|\ph(z)|> \delta} \sup_{\|f\|_{\lip^{\beta-j}}\le 1}
(1-|z|)^{J-\al} |G_j(z)||f(\ph(z)) - f(r_m\ph(z))|.
\end{aligned}
\end{equation}
The family $\{f: \|f\|_{\lip^{\beta-j}}\le 1\}$ is uniformly bounded on compact subsets of $\Dbb$,
hence $f(\ph(z)) - f(r_m\ph(z)) \rightrightarrows 0$ as $m\to \infty$ if $|\ph(z)|\le \delta$
and $\|f\|_{\lip^{\beta-j}}\le 1$.
Also,
\begin{equation}\label{e_Gj}
\sup_{z\in\Dbb} (1-|z|)^{J-\al} |G_j(z)|<\infty,
\end{equation}
since $\cph^{G_j}: \lip^{\beta-j}(\Dbb) \to \lip^{\al-J}(\Dbb)$ is bounded.
So, in \eqref{e_upper_u}, the supremum over the set $\{|\ph(z)|\le \delta\}$ tends to zero as $m\to\infty$.

For $f\in\lip^{(\beta-j)}(\Dbb)$, $\|f\|_{\lip^{\beta-j}}\le 1$, we have
\[
|f(\ph(z))| \lesssim \Om_{0, \beta-j}(|\ph(z)|) = \Om_{j, \beta}(|\ph(z)|), \quad z\in\Dbb,
\]
by Proposition~3.1 from \cite{Du11}. Therefore, Lemma~\ref{l_upper} and \eqref{e_upper_u}
guarantee that
\begin{align*}
  \|\cph^g\|_{e, \lip^\beta \to \lip^\al}
&\lesssim \sum_{j=N}^J \limsup_{|\ph(z)|\to 1-}
|G_j[g,\ph,J](z)| (1-|z|)^{J-\al} {\Om_{j, \beta}(|\ph(z)|)}\\
&\ + \sum_{j=0}^{N-1} \|\cph^{G_j}\|_{e, \lip^{\beta-j} \to \lip^{\al-J}}.
\end{align*}
We claim that the latter sum is equal to zero.
Indeed, for $j\in\{0,1,\dots, N-1\}$, let $\{f_n\}_{n=1}^\infty$
be a bounded sequence in $\lip^{\beta-j}(\Dbb)$ such that $f_n\to 0$
uniformly on compact subsets of $\Dbb$.

Since $\beta-j >0$, $f_n\rightrightarrows 0$ as $n\to \infty$ by Lemma~\ref{l_uniform0}.
Hence, \eqref{e_Gj} guarantees that
$\|\cph^{G_j} f_n\|_{\lip^{\al-J}} \to 0$ as $n\to \infty$.
Therefore,  $\cph^{G_j}: \lip^{\beta-j}(\Dbb) \to \lip^{\al-J}(\Dbb)$ is a compact operator by Lemma~\ref{l_cmp}.
So, the proof of \eqref{e_ess_up} is finished.
\end{proof}

\subsection{Lower estimates}

\begin{lem}\label{l_test}
Let $j\in\zz$, $\beta\in\Rbb$ and let $N\in\zz$ be the smallest number
such that $\beta\le N$. Assume that $N\le J$.
Then there exist $f_{w, j} \in \lip^\beta(\Dbb)$, $w\in\Dbb$, $j= N, \dots, J$, such that
$\|f_{w, j}\|_{\lip^\beta}\le 1$ and
\begin{align*}
 f_{w, j}^{(j)}(w)
&\ge C \Om_{j, \beta}(|w|),\ j=N,\dots, J,\ C=C(j, \beta)>0; \\
 \text{if}\ j\ge 1,\ \text{then}\ f_{w, j}^{(k)}(w)
&=0\ \text{for} \ k=0,\dots, j-1; \\
 \text{if} \ j<J,\ \text{then} \  |f_{w, j}^{(m)}(w)|
&\le C \Om_{m, \beta}(|w|),\ m=j+1,\dots, J,\ C=C(m, \beta)>0; \\
 f_{w,j}(z)
&\to 0\ \textrm{uniformly on compact subsets of}\ \Dbb\ \text{as}\ |w|\to 1.
\end{align*}
\end{lem}
\begin{proof}
For $w\in\Dbb$, define the functions $f_{w, j}$ as follows:
\begin{align*}
\text{if}\ \beta<N,\ \text{then}\ f_{w, N}(z)
&= \frac{(1-|w|)(z-w)^N}{(1-z\wb)^{N-\beta+1}}, \quad z\in\Dbb; \\
 \text{if} \ \beta=N,\ \text{then} \  f_{w, N}(z)
&= \frac{\left(\log\frac{e}{1-z\wb}\right)^2
(z-w)^N}{\log\frac{e}{1-|w|}}, \quad z\in\Dbb; \\
\text{if}\ N<J,\ \text{then}\ f_{w, j}(z)
&= \frac{(1-|w|)(z-w)^j}{(1-z\wb)^{j-\beta+1}}, \quad
z\in\Dbb,\ j=N+1,\dots, J.
\end{align*}
To show that the above functions have the required properties, it suffices
to use direct computations and the estimate $|z-w|\le |1-z\wb|$, $z, w\in\Dbb$.
In fact, we have $\|f_{w,j}\|_{\lip^\beta} \le C(j, \beta)$, so
we obtain functions in the unit ball of $\lip^\beta(\Dbb)$ after appropriate normalization.

Also, the following observation is useful in the only nonstandard case $\beta= N = j$:
to prove the upper estimates for $|f_{w, j}^{(m)}(w)|$ with $m=N+1, \dots, J$,
it suffices to show that $\|f_{w, N}\|_{\lip^N} \le C$, that is,
\[
\left|f_{w, N}^{(N+1)} \right| \le \frac{C}{1-|z|} = C \Om_{N, N+1}(|z|), \quad z\in \Dbb.
\]
Indeed, the above inequality implies the required upper estimates for $m=N+2, \dots, J$.
\end{proof}

\begin{prop}\label{p_ess}
Let $\al, \beta \in\Rbb$, $J\in\zz$, $J>\al$ and
let $N\in \zz$ be the minimal number such that $N\ge \beta$.
Let the weight functions $\Om_{j, \beta}$, $j=0,1,\dots, J$, be those defined by \eqref{e_Omjb}.
Assume that the weighted composition operator $\cph^g: \lip^\beta(\Dbb) \to \lip^\al(\Dbb)$
is bounded.
Then
\begin{equation}\label{e_ess_low}
\aligned
\|\cph^g &\|_{e, \lip^\beta \to \lip^\al} \\
&\gtrsim
\max_{N\le j \le J} \limsup_{|\ph(z)|\to 1-}
|G_j[g,\ph,J](z)| (1-|z|)^{J-\al} {\Om_{j, \beta}(|\ph(z)|)}.
\endaligned
\end{equation}
\end{prop}
\begin{proof}
Clearly, we may assume that $J\ge N$.
We start with $j=J$ in the right-hand side of \eqref{e_ess_low}; for smaller parameters $j$
we will proceed by induction.

So, let $\{z_n\}\subset\Dbb$ be a sequence such that $|\ph(z_n)|\to 1-$ as $n\to \infty$.
Let the functions $f_{n, J} = f_{\ph(z_n), J}$, $\|f_{n,J}\|_{\lip^\beta}\le 1$, be those provided by Lemma~\ref{l_test}.
Given a compact operator $Q: \lip^\beta(\Dbb) \to \lip^\al(\Dbb)$, we have
$\|Q f_{n, J}\|_{\lip^\al} \to 0$ as $n\to \infty$; therefore,
using the property $f_{n, J}^{(k)}(\ph(z_n))=0$, $k=0,1,\dots, J-1$, we obtain
\begin{align*}
  \|\cph^g -Q\|_{\lip^\beta\to \lip^\al}
  &\ge \limsup_{n\to \infty}\|\cph^g f_{n, J}\|_{\lip^\al} - \limsup_{n\to\infty} \|\cph^g f_{n, J}\|_{\lip^\al}\\
    &\ge \limsup_{n\to \infty} |G_J[g,\ph,J](z_n)|  |f_{n, J}^{(J)}(\ph(z_n))| (1-|z_n|)^{J-\al} \\
    &\ge C\limsup_{n\to \infty} |G_J[g,\ph,J](z_n)|  \Om_{J, \beta}(|\ph(z_n)|) (1-|z_n|)^{J-\al}.
\end{align*}
Since $Q$ is an arbitrary compact operator and $\{z_n\}_{n=1}^\infty \subset \Dbb$ is an arbitrary sequence
such that $|\ph(z_n)|\to 1-$ as $n\to\infty$, we conclude that
\[
\|\cph^g\|_{e, \lip^\beta \to \lip^\al} \gtrsim
\limsup_{|\ph(z)|\to 1-} |G_j[g,\ph,J](z)|  {\Om_{J, \beta}(|\ph(z)|)} (1-|z|)^{J-\al}.
\]
In other words, the required lower estimate holds for $j=J$; in particular, the proof is finished
if $J=N$.

If $J>N$, then we argue by induction. Namely, let $J-1 \ge j \ge N$.
Assume, as the induction hypothesis, that
\begin{equation}\label{e_ind_j1}
\aligned
C_{j+1}
&\|\cph^g\|_{e, \lip^\beta \to \lip^\al} \\
&\ge
\sum_{m=j+1}^{J} \limsup_{|\ph(z)|\to 1-}
|G_m[g,\ph,J](z)| (1-|z|)^{J-\al} {\Om_{m, \beta}(|\ph(z)|)}.
\endaligned
\end{equation}
Let $\{z_n\}\subset\Dbb$ be a sequence such that $|\ph(z_n)|\to 1-$ as $n\to \infty$.
Applying Lemma~\ref{l_test}, we obtain test functions $f_{n, j} = f_{\ph(z_n), j}$.
In particular, we have
\[
|f^{(m)}_{n, j}(z)| \lesssim \Om_{m, \beta}(|z|),\quad z\in \Dbb,\ m= j+1, \dots, J.
\]
So, without loss of generality, we assume that
\begin{equation}\label{e_konst_j1}
|f^{(m)}_{n, j}(z)| \leq k_{j+1} \Om_{m, \beta}(|z|), \quad z\in \Dbb,
\end{equation}
for a constant $k_{j+1}>0$ and for $m= j+1, \dots, J$.

Now, for a compact operator $Q: \lip^\beta(\Dbb) \to \lip^\al(\Dbb)$, we have
\begin{align*}
  \|\cph^g -Q
  &\|_{\lip^\beta\to \lip^\al}
  \ge \limsup_{n\to \infty}\|\cph^g f_{n, j}\|_{\lip^\al} - \limsup_{n\to\infty} \|\cph^g f_{n, j}\|_{\lip^\al}\\
    &\ge \limsup_{n\to \infty} |G_j[g,\ph,J](z_n)|  |f_{n, j}^{(j)} (\ph(z_n))| (1-|z_n|)^{J-\al} \\
    &\ -\sum_{m=j+1}^J \limsup_{n\to \infty} |G_m[g,\ph,J](z_n)| |f_{n, j}^{(m)} (\ph(z_n))| (1-|z_n|)^{J-\al} \\
  &\ge C \limsup_{n\to \infty} |G_j[g,\ph,J](z_n)|  \Om_{j, \beta}(|\ph(z_n)|) (1-|z_n|)^{J-\al} \\
    &\ -k_{j+1}\sum_{m=j+1}^J \limsup_{n\to \infty} |G_m[g,\ph,J](z_n)|
       {\Om_{m, \beta}(|\ph(z_n)|)} (1-|z_n|)^{J-\al}
\end{align*}
by \eqref{e_konst_j1} and Lemma~\ref{l_test}.

Since $\{z_n\}_{n=1}^\infty \subset \Dbb$ is an arbitrary sequence
such that $|\ph(z_n)|\to 1-$ as $n\to\infty$, we obtain
\begin{align*}
\|\cph^g\|_{e, \lip^\beta \to \lip^\al}
&\ge \limsup_{|\ph(z)|\to 1-} |G_j[g,\ph,J](z)|  {\Om_{j, \beta}(|\ph(z)|)} (1-|z|)^{J-\al} \\
- &k_{j+1} \sum_{m=j+1}^J \limsup_{|\ph(z)|\to 1-} |G_m[g,\ph,J](z)|
       {\Om_{m, \beta}(|\ph(z)|)} (1-|z|)^{J-\al}.
\end{align*}
Therefore, the induction hypothesis \eqref{e_ind_j1} guarantees that
\begin{align*}
  (C_{k+1} k_{j+1} + 1)\|\cph^g\|_{e, \lip^\beta \to \lip^\al}
&\ge  \\
 C \limsup_{|\ph(z)|\to 1-}
&|G_j[g,\ph,J](z)|  {\Om_{j, \beta}(|\ph(z)|)} (1-|z|)^{J-\al}.
\end{align*}
In other words, the induction hypothesis \eqref{e_ind_j1} holds with $j$ in the place of $j+1$.
So, the induction construction proceeds and we obtain the required lower estimate for all $j$, $N\le j \le J$.
\end{proof}

\subsection{Main result}

We will use the following auxiliary result:
\begin{prop}[see {\cite[Theorem~2.4]{HKLRS}}]\label{p_zn_ess}
Let $\om, \nu$ be weights on $[0,1)$, $g\in \hol(\Dbb)$ and let $\ph: \Dbb\to \Dbb$ be a holomorphic mapping.
Then
\[
\limsup_{|\ph(z)|\to 1}\frac{|g(z)|\widetilde{\om}(|\ph(z)|)}{\nu(|z|)} =
\limsup_{n\to\infty}
\frac{\|g\ph^n\|_{H^\infty_\nu}}{\|z^n\|_{H^\infty_\om}}.
\]
\end{prop}

\begin{theorem}\label{t_ess}
Let $\al, \beta \in\Rbb$, $J\in\zz$, $J>\al$ and
let $N\in \zz$ be the minimal number such that $N\ge \beta$.
Let the weight functions $\Om_{j, \beta}$, $j=0,1,\dots, J$, be those defined by \eqref{e_Omjb}.
Assume that the weighted composition operator $\cph^g: \lip^\beta(\Dbb) \to \lip^\al(\Dbb)$
is bounded.
Then
\begin{equation}\label{e_ess}
\begin{aligned}
  \|\cph^g\|_{e, \lip^\beta \to \lip^\al}
&\asymp \max_{N\le j \le J} \limsup_{|\ph(z)|\to 1-}
\frac{|G_j[g,\ph,J](z)| \Om_{j, \beta}(|\ph(z)|)}{(1-|z|)^{\al-J}} \\
&=\max_{N\le j \le J} \limsup_{n\to\infty}
\frac{\|G_j[g,\ph,J] \ph^n\|_{\lip^{\al-J}}}
{\|z^n\|_{H^\infty_{\Om_{j, \beta}}}}
\end{aligned}
\end{equation}
In particular, $\cph^g: \lip^\beta(\Dbb) \to \lip^\al(\Dbb)$ is compact if $N>J$.
\end{theorem}
\begin{proof}
Firstly, combining Propositions~\ref{p_upper} and \ref{p_ess}, we have the required equivalence.
Secondly, if $N\le j \le J$, then $\Om_{j, \beta}$ is a weight function and $\widetilde{\Om}_{j, \beta} = \Om_{j, \beta}$; hence, Proposition~\ref{p_zn_ess} applies
 with $\om = \Om_{j, \beta}$, $\nu(t) =(1-t)^{\al-J}$ and $g=G_j[g, \ph, J]$.
\end{proof}

\bibliographystyle{amsplain}

\end{document}